\documentclass[10pt]{siamltex}
\usepackage{amsmath}
\usepackage{amsfonts}
\usepackage{amssymb}
\usepackage{graphicx}
\usepackage{color}
\usepackage{enumerate}
\usepackage{bm}
\usepackage{stmaryrd}
\usepackage{comment}

\newtheorem{example}[theorem]{Example}

\newtheorem{remark}[theorem]{Remark}

\newcommand{\bea}{\begin{eqnarray}}
\newcommand{\eea}{\end{eqnarray}}
\newcommand{\beas}{\begin{eqnarray*}}
\newcommand{\eeas}{\end{eqnarray*}}

\newcommand{\mcK}{\ensuremath{\mathcal{K}}}

\newcommand{\bA}{\ensuremath{\bm{A}}}
\newcommand{\bv}{\ensuremath{\bm{v}}}
\newcommand{\bx}{\ensuremath{\bm{x}}}
\newcommand{\bz}{\ensuremath{\bm{z}}}
\newcommand{\by}{\ensuremath{\bm{y}}}

\newcommand{\be}{\ensuremath{\bm{e}}}
\newcommand{\oS}{\ensuremath{\overline{S}}}
\newcommand{\ob}{\ensuremath{\overline{b}}}

\newcommand{\oa}{\ensuremath{\overline{a}}}
\newcommand{\cU}{\ensuremath{\mathcal{U}}}
\newcommand{\cK}{\ensuremath{\mathcal{K}}}

\graphicspath{{./figs/}}

\input{tcilatex}

\title{A class of null space conditions for sparse recovery via nonconvex, non-separable minimizations}
\author{
Hoang~Tran\thanks{Department of Computational and Applied Mathematics, Oak Ridge National Laboratory, 1 Bethel Valley Road, P.O. Box 2008, Oak Ridge TN 37831-6164. email: \texttt{tranha@ornl.gov}. }
\and Clayton~Webster\thanks{Department of Mathematics, University of Tennessee, 1403 Circle Drive, Knoxville, TN 37916 and Department of Computational and Applied Mathematics, Oak Ridge National Laboratory, 1 Bethel Valley Road, P.O. Box 2008, Oak Ridge TN 37831-6164. email: \texttt{webstercg@ornl.gov}.}\
}

\begin{document}
\maketitle

\begin{abstract}
For the problem of sparse recovery, it is widely accepted that nonconvex minimizations are better than $\ell_1$ penalty in enhancing the sparsity of solution. However, to date, the theory verifying that nonconvex penalties outperform (or are at least as good as) $\ell_1$ minimization in exact, uniform recovery has mostly been limited to separable cases. In this paper, we establish general recovery guarantees through null space conditions for {nonconvex, non-separable} regularizations, which are slightly less demanding than the standard null space property for $\ell_1$ minimization.  
\end{abstract}

\begin{keywords}
nonconvex optimization, null space property, sparse recovery, majorization theory
\end{keywords}

\begin{AMS}
94A12, 94A15, 90C26
\end{AMS}

\section{Introduction}
\label{sec:intro}


This paper is concerned with the reconstruction of a sparse signal $\bx \in \mathbb{R}^N$ from relatively few observed data $\by \in \mathbb{R}^m$. More precisely, we recover the unknown vector $\bx \in \mathbb{R}^N$ from the system 
\begin{align}
\label{problem:und_system}
\bA \bz  = \by, 
\end{align}
given the matrix $\bA \in \mathbb{R}^{m\times N}$ $(m\ll N)$, and using linear measurements $\by = \bA \bx$. In general, the system \eqref{problem:und_system} is underdetermined and has infinitely many solutions. With the additional acknowledgement that the unknown signal is sparse, which is the case in several contexts such as compressed sensing \cite{CRT06,Donoho06}, statistics \cite{HastieTibshiraniWainwright15}, and uncertainty quantification \cite{DO11,RS14,DexterTranWebster17}, we search for the sparsest solution of \eqref{problem:und_system} only. It is natural to reconstruct $\bx$ via the $\ell_0$ minimization problem
\begin{align}
\label{problem:l0}
\min_{\bz \in \mathbb{R}^N} \|\bz\|_{0}, \text{ subject to } \bA \bz = \by,
\end{align}
where $\|\bz\|_0$ is the number of nonzero components in $\bz$.
However, since the locations of the nonzero components are not available, solving \eqref{problem:l0} directly requires a combinatorial search and is unrealistic in general. 

An alternative and popular approach is basis pursuit or $\ell_1$ minimization, which consists in finding the minimizer of the problem
\begin{align}
\label{problem:l1}
\min_{\bz \in \mathbb{R}^N} \|\bz\|_{1}, \text{ subject to } \bA \bz = \by.
\end{align}
The convex optimization problem \eqref{problem:l1} is an efficient relaxation for \eqref{problem:l0} and often produces sparse solutions. The sparse recovery property of $\ell_1$ minimization has been well-developed. It is known from the compressed sensing literature that if $\bA$ possesses certain properties, such as the null space property and restricted isometry property, problem \eqref{problem:l0} and its convex relaxation \eqref{problem:l1} are equivalent \cite{FouRau13}.

Although the $\ell_1$ minimization technique has been used in a wide variety of problems, it is not able to reconstruct the sparsest solutions in many applications. As such, several nonconvex regularizations have been applied to improve the recovery performance. These penalties are generally closer to the $\ell_0$ penalty than the $\ell_1$ norm, thus it is reasonable to expect that nonconvex minimizations can further enhance the sparsity of the solutions. The most commonly used nonconvex penalty is probably $\ell_p$ with $0 < p < 1$ \cite{Chartrand07,FoucartLai09}, which interpolates between $\ell_0$ and $\ell_1$. Other well-known nonconvex methods in the literature include Smoothly Clipped Absolute Deviation (SCAD) \cite{FanLi01}, capped $\ell_1$ \cite{NIPS2008_3526,ShenPanZhu12} and transformed $\ell_1$ \cite{LvFan09,ZhangXin16}. All of these penalties are separable, in the sense that they are the sum of the penalty functions applied to each individual component of the vector. Recently, many non-separable regularizations are also considered, for instance, iterative support detection \cite{WangYin10}, 
$\ell_1 - \ell_2$ \cite{EsserLouXin13, YinLouHeXin15,YanShinXiu16}, two-level $\ell_1$ \cite{HuangLiuShi-et-al15} and sorted $\ell_1$ \cite{HuangShiYan15}. While nonconvex penalties are generally more challenging to minimize, they have been shown to reconstruct the sparse target signals from significantly fewer measurements in many computational test problems. 

Yet, to date, the theoretical evidence that nonconvex penalties are superior to (or at least, not worse than) $\ell_1$ optimization in uniform reconstruction of sparse signals has not been fully developed. For some regularizations such as $\ell_p$ ($0<p<1$) \cite{FoucartLai09}, iterative support detection \cite{WangYin10}, recovery guarantees are available, established via variants of restricted isometry property (RIP), and proved to be less restrictive than that of $\ell_1$. However, such guarantees remain elusive or suboptimal for many others, especially non-separable penalties. Let us elucidate this point by an examination of the null space properties recently developed for $\ell_1 - \ell_2$ \cite{YinLouHeXin15,YanShinXiu16} and two-level $\ell_1$ \cite{HuangLiuShi-et-al15}. 
For any $\bm{v} = (v_1,\ldots,v_N)\in \mathbb{R}^N$ and $S\subset \{1,\ldots,N\}$, denote by  $\oS$ the complement of $S$ in $\{1,\ldots,N\}$ and $\#(S)$ the cardinality of $S$, then the null space property of the measurement matrix $\bA$, introduced in \cite{CohenDahmenDeVore09}, can be stated as follows. 
\begin{definition}[Null space property]
For the matrix $\bA \in \mathbb{R}^{m\times N}$, the null space property is given by:
\begin{gather}
\label{NSP:l1}
\tag{\fontfamily{cmss}\selectfont NSP}
\begin{aligned}
\ker(\bA)\!\setminus\!\{\mathbf{0}\}\subset \bigg\{ \bv\! \in \! \mathbb{R}^N: 
\|\bv_S\|_1  < \|\bv_{\overline{S}}\|_1,\, \forall S \!\subset \!\{1,\ldots,N\}\text{ with }\#(S)\le s\bigg\}.
\end{aligned}
\end{gather}
\end{definition}%
It is well-known that \eqref{NSP:l1} is the necessary and sufficient condition for the successful reconstruction using $\ell_1$ minimization \cite{CohenDahmenDeVore09,FouRau13}, thus, stronger recovery properties, which are desirable and expected for nonconvex minimizations, would essentially allow $\ker(\bA)$ to be contained in a larger set than that in \eqref{NSP:l1}. The arguments in \cite{YinLouHeXin15,YanShinXiu16,HuangLiuShi-et-al15} are, however, at odds with this observation. Therein, the RIP were developed so that
\begin{align*}
&  \ker(\bm{A})\! \setminus\! \{\mathbf{0}\} \subset   \bigg\{ \bm{v} \in \mathbb{R}^N\!:\!
\|\bm{v}_S\|_1  < \|\bm{v}_{H\cap {\oS}}\|_1,\, \forall S\!: \#(S)\le s,\, \forall H\!:\#(H) = \lfloor N/2  \rfloor\! \bigg\},
\\
& \text{and} \ \ker(\bA)\setminus\{\mathbf{0}\}\subset  \bigg\{ \bm{v} \in \mathbb{R}^N\!:\!
\|\bm{v}_S\|_1 + \|\bm{v}_S\|_2 + \|\bm{v}_{\oS}\|_2  < \|\bm{v}_{{\oS}}\|_1,\, \forall S\!: \#(S)\le s\! \bigg\}, 
\end{align*}
respectively for two-level $\ell_1$ and $\ell_1 - \ell_2$ penalties. As then $\ker({\bA})$ was restricted to strictly smaller sets than as in $\ell_1$ case, the acquired RIPs were inevitably more demanding.   

In this paper, we establish new uniform recovery guarantees for a general class of nonconvex, possibly non-separable minimizations, which are superior to or at least as good as $\ell_1$. More specifically, we consider the nonconvex optimization problem in general form 
\begin{align}
\label{problem:PR}
\tag{${\text{\fontfamily{cmss}\selectfont P}}_{\text{\fontfamily{cmss}\selectfont R}}$}
\underset{\bm{z}\in \mathbb{R}^N}{\text{minimize}}  \ \ R(\bm{z}) \text{ subject to }\bm{A}\bm{z} = \bm{A}\bm{x}, 
\end{align}
and, under some mild assumptions of $R$ (applicable to most nonconvex penalties considered in the literature), derive null space conditions for the exact reconstruction via \eqref{problem:PR}, which are less demanding or identical to the standard conditions required by the $\ell_1$ norm. 
Our main achievement in this paper is: for the regularizations that are concave, non-separable and symmetric\footnote{The precise definitions of \textit{separable}, \textit{concave} and \textit{symmetric} are presented in Section \ref{sec:background}.} (such as $\ell_1 - \ell_2$, two-level $\ell_1$, sorted $\ell_1$), \eqref{NSP:l1} is sufficient for the exact recovery of all $s$-sparse signals. Furthermore, in many cases, an improved variant of \eqref{NSP:l1} is enough, i.e., 
\begin{gather}
\label{iNSP:nonsep}
\tag{\fontfamily{cmss}\selectfont iNSP}
\begin{aligned}
\ker(\bA)\setminus\{\mathbf{0}\}\subset \bigg\{ \bv \in \mathbb{R}^N: 
\|\bv_S\|_1  \le \|\bv_{\overline{S}}\|_1,\, \forall S\mbox{ with }\#(S)\le s\bigg\}.
\end{aligned}
\end{gather}
One distinct aspect of these results, as we shall see, is that they do \textit{not} have fixed support version. As such, our analysis requires technical arguments beyond the standard approach, that is, first deriving and then combining all null space properties for the recovery of vectors supported on fixed sets of same cardinality. 

To compare with the better known case where separable property is assumed, we revisit the necessary and sufficient condition for the exact recovery for concave and separable regularizations. The generalized null space property 
\begin{gather}
\label{iNSP:sep}
\tag{\fontfamily{cmss}\selectfont gNSP}
\begin{aligned}
\ker(\bA)\setminus\{\mathbf{0}\}\subset \bigg\{ \bv \in \mathbb{R}^N: 
R(\bv_S)  < R(\bv_{\overline{S}}),\, \forall S\mbox{ with }\#(S)\le s\bigg\}
\end{aligned}
\end{gather}
can be established from fixed support null space conditions, therefore is a routine extension of several similar results for specific penalties \cite{FoucartLai09, FouRau13}. This property can also be found in \cite{GribonvalNielsen07}. However, \eqref{iNSP:sep} is not automatically less restrictive than the standard \eqref{NSP:l1}, and we verify that this is the case only if the regularization is also symmetric. Penalties that can be treated in this setting include $\ell_p$, SCAD, transformed $\ell_1$, capped $\ell_1$.



\subsection{Related works}
This paper examines the recovery of sparse signals by virtue of \textit{global minimizers} of nonconvex problems. These are the best solutions one can acquire via the considered nonconvex regularizations, regardless of the numerical procedures used to realize them.  Therefore, our  
results serve as a benchmark for the performance of concrete algorithms. Often in practice, one can only obtain the \textit{local minimizers} of nonconvex problems. The theoretical recovery properties via local minimizers, attached to specific numerical schemes, has also gained considerable attention in the literature. In \cite{NIPS2008_3526, ZhangT10}, a multi-stage convex relaxation scheme was developed for solving problems with nonconvex objective functions, with a focus on capped-$\ell_1$. Theoretical error bound was established for fixed designs showing that the local minimum solution obtained by this procedure is superior to the global solution of the standard $\ell_1$. In \cite{YinXin16}, $\ell_{1} -\ell_{2}$ minimization was proved not worse than $\ell_1$, based on a difference of convex function algorithm (DCA), which is an iterative procedure and returns $\ell_1$ solution in the first step. 

Generalized conditions for nonconvex penalties were also established in \cite{FanLi01,LvFan09}, under which three desirable properties of the regularizations - unbiasedness, sparsity, and continuity - are fulfilled. Therein, the properties of the local minimizers and the sufficient conditions for a vector to be the local minimizer were analyzed with a unified approach, and specified for SCAD and transformed $\ell_1$ (referred to as SICA in \cite{LvFan09}). We remark that this framework applies to separable penalties only.  

Finally, sorted $\ell_1$ is a nonconvex method recently introduced in \cite{HuangShiYan15}. This method generalizes several nonconvex approaches, including iterative hard thresholding, two-level $\ell_1$, truncated $\ell_1, $ and small magnitude penalized (SMAP). 

\subsection{Organization}
The remainder of this paper is organized as follows. In Section \ref{sec:background}, we describe the general nonconvex, non-separable minimization problem, some examples, and provide the necessary background results. In Section \ref{sec:nonseparable}, we prove our main results on the null space properties for the exact reconstruction via non-separable penalties. The recovery conditions for separable  
penalties are discussed in Section \ref{sec:separable}. Finally, concluding remarks are given in Section \ref{sec:conclusion}.

\section{Nonconvex, non-separable minimization problem}
\label{sec:background}
Throughout this paper, we denote $\mathcal{U}= [0,\infty)^N$, and for $\bm{z} = (z_1,\ldots,z_N)\in \mathbb{R}^N$, let $|\bz| =  (|z_1|,\ldots,|z_N|)$ with 
$
|z|_{[1]} \ge \ldots \ge |z|_{[N]}
$
the components of $|\bz|$ in decreasing order. If $\bz$ has nonnegative components, i.e., $\bz \in \cU$, we simply write 
$$
z_{[1]} \ge \ldots \ge z_{[N]}. 
$$ 
We call a vector $\bz \in \mathbb{R}^N$: {equal-height} if all nonzero coordinates of $\bz$ have the same magnitude; and  $s$-sparse if it has at most $s$ nonzero coefficients. Also, $\be_j$ is the standard basis vector with a $1$ in the 
$j$-th coordinate and $0$'s elsewhere. The $j$-th coordinate of a vector $\bz\in \mathbb{R}^N$ is often denoted simply by $z_j$, but at some places, we also use the notation $(\bz)_j$.

Recall the nonconvex optimization problem of interest is given by 
\begin{align}
\label{problem:PR_copy}
\tag{${\text{\fontfamily{cmss}\selectfont P}}_{\text{\fontfamily{cmss}\selectfont R}}$}
\underset{\bm{z}\in \mathbb{R}^N}{\text{minimize}}  \ \ R(\bm{z}) \text{ subject to }\bm{A}\bm{z} = \bm{A}\bm{x}. 
\end{align}
We define the following theoretical properties of the penalty $R$ described in \eqref{problem:PR_copy}.
\begin{definition}
\label{def:penalty_prop}
Let $R$ be a mapping from $\mathbb{R}^N$ to $[0,\infty)$ satisfying $R(z_1,\ldots,z_N)$ $= R(|z_1|,\ldots,|z_N|)$, for all  $\bm{z} = (z_1,\ldots,z_N)\in\mathbb{R}^N$.  
\vspace{.2cm}
\begin{itemize}
\item $R$ is called \textbf{separable} on $\mathbb{R}^N$ if there exist functions $r_j: \mathbb{R} \to [0,\infty),\, j\in \{1,\ldots, N\}$ such that for every $\bz \in \mathbb{R}^N$, $R$ can be represented as 
\begin{align}
\label{def:separable}
R(\bz) = \sum_{j=1}^N r_j(z_j).
\end{align}
If $R$ cannot be written in the form \eqref{def:separable}, we say $R$ is \textbf{non-separable} on $\mathbb{R}^N$. 
\item $R$ is called \textbf{symmetric} on $\mathbb{R}^N$ if for every $\bm{z} \in \mathbb{R}^N$ and every permutation $(\pi(1),\ldots,\pi(N))$ of $(1,\ldots,N)$:
\begin{align}
\label{def:symmetric}
R(z_{\pi(1)},\ldots,z_{\pi(N)}) = R(\bm{z}).
\end{align}
\item $R$ is called \textbf{concave} on $\mathcal{U}$ if for every $\bm{z},\bm{z}' \in \mathcal{U}$ and $0\le \lambda \le 1$: 
\begin{align}
\label{def:concave}
R(\lambda \bm{z} + (1-\lambda )\bm{z}') \ge \lambda R(\bm{z}) + (1-\lambda) R(\bm{z}').
\end{align}
\item $R$ is called \textbf{increasing} on $\mathcal{U}$ if for every $\bm{z},\bm{z}'\in \mathcal{U}$, $\bm{z}\ge \bm{z}'$ then 
\begin{align}
\label{def:increasing}
R(\bm{z}) \ge R(\bm{z}'). 
\end{align}
Here, $\bm{z}\ge \bm{z}'$ means $z_j \ge z'_j,\, \forall 1\le j\le N$.
\end{itemize}
\end{definition}
%
%
%
%
The theory for uniform recovery developed herein is applicable for all concave, non-separable and symmetric penalties. In particular, our present conditions unify and improve the existing conditions for $\ell_1 - \ell_2$ and two-level $\ell_1$. On the other hand, we provide, for the first time, the theoretical requirements for the uniform recovery of sorted $\ell_1$. 
\begin{example}[Concave, non-separable and symmetric penalties]
\label{example:ncv_func}

\begin{enumerate}
\itemsep10pt
\item $\ell_1 - \ell_2$: \qquad \qquad \qquad \quad \,  \ \ $R_{\ell_1 - \ell_2}(\bm{z}) = \|\bm{z}\|_1 -  \|\bm{z}\|_2$, {\cite{EsserLouXin13, YinLouHeXin15}}.
\item Two-level $\ell_1$: \qquad \qquad \quad \quad $R_{2\ell_1}(\bm{z}) = \rho \sum\limits_{j\in J(\bm{z})} |z_j| +\! \sum\limits_{j\in J(\bm{z})^c} |z_j|$, {\cite{HuangLiuShi-et-al15}}.
\\
Here, $0\le \rho < 1$ and $J({\bm{z}})$ is the set of largest components of $|z_j|$. 
\item Sorted $\ell_1$: \qquad \qquad \qquad \quad \ {$R_{s\ell_1}(\bm{z}) = \beta_1|{z}|_{[1]} +   \ldots + \beta_N|{z}|_{[N]}$,  \cite{HuangShiYan15}}.  

 \vspace{.1cm}
 Here, $0\le \beta_1\le \ldots \le \beta_N$ and $|{z}|_{[1]} \ge \ldots \ge |{z}|_{[N]}$ are the components of $|\bz|$ ranked in decreasing order. 
\end{enumerate}
\end{example}

We remark that our analysis does not cover non-concave or non-symmetric regularizations. However, unlike their concave and symmetric counterparts, \eqref{NSP:l1} may not be sufficient to guarantee the uniform and sparse reconstruction with these penalties in general. Therefore, non-concave or non-symmetric regularizations are not necessarily better than $\ell_1$ minimization in exact, uniform recovery, in the sense that for some sampling matrices, $\ell_1$ can successfully recover all $s$-sparse vectors, while a non-concave or non-symmetric penalty fails to do so. 
A well-known example of non-concave penalties which are less efficient than $\ell_1$ is $\ell_p$ with $p>1$. Another interesting example is $\ell_1/\ell_2$ \cite{EsserLouXin13,YinEsserXin14}, a non-concave \textit{and} non-convex regularization. Whether there exists a null space property less restrictive than \eqref{NSP:l1} for this penalty is an open question. 
Non-symmetric regularizations, on the other hand, do not always recover the sparsest vectors due to their preference for some components over others. An example of a non-symmetric penalties is given in Section \ref{sec:separable}.  
\subsection{Properties of penalty functions}
Next, we present a few necessary supporting results for the penalty functions of interest. These results are relatively well-known, so will be provided here without proofs. 
\begin{lemma}
\label{lem:incr_prop}
Let $R$ be a map from $\mathbb{R}^N$ to $[0,\infty)$. If $R$ is concave on $\cU$, then $R$ is increasing on $\cU$. 
\end{lemma}

Note that if $R:\mathbb{R}^N \to [0,\infty)$ satisfies $R(z_1,\ldots,z_N) = R(|z_1|,\ldots,|z_N|)$, $\forall \bm{z} = (z_1,\ldots,z_N)\in\mathbb{R}^N$ and is increasing on $\mathcal{U}$, then
\begin{align}
\label{inc_prop}
R(\bm{z}) \ge R(\bm{z}')\ \text{for all } \bz, \bz' \in \mathbb{R}^N \text{ with }  |z_j| \ge |z'_j|,\, \forall 1\le j\le N. 
\end{align}
$R$ is therefore increasing in  the whole space $\mathbb{R}^N$ in the sense of \eqref{inc_prop}. We will use both terms ``increasing on $\cU$'' and ``increasing on $\mathbb{R}^N$'' interchangeably in the sequel.

To establish the generalized conditions for successful sparse recovery in the non-separable case, we employ the concept of \textit{majorization}. This notion, defined below, makes precise and rigorous the idea that the components of a vector are ``more (or less) equal'' than those of another.

\begin{definition}[Majorization, \cite{MarshallOlkinArnold11}]
For $\bz,\bz' \in \cU$, $\bz$ is said to be majorized by $\bz'$, denoted by $\bz \prec \bz' $, if
\begin{align}
\label{def:majorization}
\begin{cases}
\sum\limits_{j=1}^n z_{[j]} \le \sum\limits_{j=1}^n z'_{[j]},\ \ \ n = 1,\ldots, N-1,
\\
\sum\limits_{j=1}^N z_{[j]} = \sum\limits_{j=1}^N z'_{[j]}. 
\end{cases}
\end{align}
Given condition \eqref{def:majorization}, we also say $\bz'$ majorizes $\bz$ and denote $\bz' \succ \bz $. 
\end{definition}

As a simple example of majorization, we have
\begin{gather*}
\begin{aligned}
&\left(\frac16,\frac16,\frac16,\frac16,\frac16,\frac16\right) \prec \left(\frac14,\frac14,\frac14,\frac14,0,0\right) \prec \left(\frac38,\frac14,\frac14,\frac18,0,0\right) 
\\
&\qquad \qquad \prec \left(\frac12,\frac12,0,0,0,0\right) \prec  \left(\frac23,\frac13,0,0,0,0\right) \prec \left(1,0,0,0,0,0\right). 
\end{aligned}
\end{gather*}
Loosely speaking, a sparse vector tends to majorize a dense one with the same $\ell_1$ norm. On the other hand, a sparse-promoting penalty function should have small values at sparse signals and larger values at dense signals. One may think that a penalty function which reverses the order of majorization would promote sparsity, in particular, outperform $\ell_1$. We will show in the next sections that this intuition is indeed correct, but first, let us clarify that all symmetric and concave penalty functions considered are order-reversing\footnote{Functions that reverse the order of majorization are often referred to as Schur-concave functions, see, e.g., \cite{MarshallOlkinArnold11}.}, see \cite[Chapter 3.C]{MarshallOlkinArnold11} and \cite[Remark II.3.7]{Bhatia97}.

\begin{lemma}
\label{lemma:majorize}
Let $R$ be a function from $\mathbb{R}^N$ to $[0,\infty)$ satisfying $R(\!z_1,\ldots,\!z_N\!)$ $ = R(|z_1|,\ldots,|z_N|),\ \forall \bm{z} = (z_1,\ldots,z_N)\in\mathbb{R}^N$. If $R$ is symmetric on $\mathbb{R}^N$ and concave on $\cU$, then $R$ reverses the order of majorization: 
\begin{align}
\label{eq:schur-ccv}
R(\bz) \ge R(\bz') \text{ for all }\bz,\bz' \in \cU\text{ with }\bz \prec \bz'.
\end{align}
\end{lemma}
\section{Exact recovery of sparse signals via non-separable penalties}
\label{sec:nonseparable}
In this section, we prove that concave, non-separable and symmetric regularizations are superior to $\ell_1$ in sparse, uniform recovery. This setting applies to penalties such as two-level $\ell_1$, sorted $\ell_1$, and $\ell_1 - \ell_2$. Our main result is given below.
\begin{theorem}
\label{main_theorem}
Let $N>1$, $s\in \mathbb{N}$ with $1 \le  s <  N/2 $, and $\bA$ be an $m\times N$ real matrix. Consider the problem \eqref{problem:PR}, where $R$ is a function from $\mathbb{R}^N$ to $[0,\infty)$ satisfying $R(z_1,\ldots,z_N) = R(|z_1|,\ldots,|z_N|),\ \forall \bm{z} = (z_1,\ldots,z_N)\in\mathbb{R}^N$, symmetric on $\mathbb{R}^N$, and concave on $\mathcal{U}$.  
\begin{enumerate}[\ \ \ i)]
\item If 
\begin{align}
 \label{assump1}
 \tag{${\text{\fontfamily{cmss}\selectfont R}_{1}}$}
 R(z_1,\ldots,z_s, z_{s+1},{0,\ldots,0}) > R(z_1,\ldots, z_s,0,\ldots,0),\ \forall z_1, \ldots, z_{s+1} > 0,
 \end{align}
  then every $s$-sparse vector $\bm{x}\in \mathbb{R}^N$ is the unique solution to \eqref{problem:PR} provided that 
the null space property \eqref{NSP:l1} is satisfied. In this sense, {\eqref{problem:PR} is at least as good as $\ell_1$-minimization}.  
\item If 
\begin{gather}
  \label{assump2}
 \tag{${\text{\fontfamily{cmss}\selectfont R}_{2}}$}
 \begin{aligned}
 R(z_1,\ldots,z_{s-1},z_s, z_{s+1},  0,\ldots,  0) >\, R( z_1,& \ldots,z_{s-1},z_s + z_{s+1},{0,\ldots,0}),\ 
 \\
 & \forall z_1, \ldots, z_{s+1} > 0, \notag 
 \end{aligned}
 \end{gather}
 then every $s$-sparse vector $\bm{x}\in \mathbb{R}^N$ (except equal-height vectors) is the unique solution to \eqref{problem:PR} provided that the improved null space property \eqref{iNSP:nonsep} is satisfied. The recovery guarantee of \eqref{problem:PR} therefore is better than that of $\ell_1$-minimization. 
\end{enumerate}
\end{theorem}
{

It is worth emphasizing that Theorem \ref{main_theorem} does not have a fixed support version. More specifically, it may be tempting to think that \eqref{NSP:l1} can be proved to be a sufficient condition for non-separable minimizations by the same mechanism as in $\ell_1$ and separable cases, i.e., combining all conditions for the recovery of vectors supported on fixed sets $S$ of cardinality $s$, i.e.,
\begin{gather}
\label{NSP-fs}
\begin{aligned}
\ker(\bA)\!\setminus\!\{\mathbf{0}\}\subset \bigg\{ \bv\! \in \! \mathbb{R}^N: 
\|\bv_S\|_1  < \|\bv_{\overline{S}}\|_1\bigg\}, 
\end{aligned}
\end{gather}
see \cite[Section 4.1]{FouRau13}. However, this strategy does not work, as we can show that unlike $\ell_1$, \eqref{NSP-fs} does not guarantee the successful recovery of vectors supported on $S$ with non-separable penalties.   
Indeed, consider the underdetermined system 
$
\bA\bz = \bA\bx,
$
where the matrix $\bA\in \mathbb{R}^{4\times 5}$ is defined as 
\[ \bA =  \left( \begin{array}{ccccc}
 1 & 0.5 & 1  & 0 & 0 \\
 1 & -0.5 & 0 & 1 & 0 \\
 0 & 0.1 & 0 & 0 & 1 \\
 1 & -1 & 0 & 0 & 0 
 \end{array} \right).
  \] 
As $\ker(\bA) = (-t,-t,3t/2,t/2,t/10)^{\top}$ satisfies \eqref{NSP-fs} with $S=\{1,2\}$, all sparse signals supported on $S$ can be exactly recovered with $\ell_1$ penalty. Consider $\ell_1 -  \ell_2$ regularization, let $\bx = (1,1,0,0,0)^\top$ supported on $S$, then any solution to $\bA\bz = \bA\bx$ can be represented as $(1-t,1-t,3t/2,t/2,t/10)^{\top}$. Among those, the unique minimizer of $R_{\ell_1 - \ell_2}(\bz)$ is $\bz\! =\! (0,0,3/2,1/2,1/10)^\top$, which is different from $\bx$ and not the sparsest solution. 
}

The proof of Theorem \ref{main_theorem} is rather lengthy and is relegated to Section \ref{sec:proof}. Let us first discuss the assumptions \eqref{assump1} and \eqref{assump2}. We will see from this proof that the concavity and symmetry of the penalty function $R$ is enough to guarantee every $s$-sparse vector is a solution to \eqref{problem:PR}. For the exact recovery, we also need these solutions to be unique. Such uniqueness could be derived assuming $R$ is strictly concave, but several regularizations do not satisfy this property. Rather, we only require strict concavity (or strictly increasing property) in one direction and locally at $s$-sparse vectors, reflected in \eqref{assump1} and \eqref{assump2}. These mild conditions can be validated easily for the considered non-separable penalties; see Proposition \ref{prop:non-sep}. We note that \eqref{assump1} is weaker than \eqref{assump2}. 

On the other hand, \eqref{iNSP:nonsep} cannot guarantee the exact recovery of equal-height, $s$-sparse vectors with symmetric penalties in general. Indeed, it is possible that $\ker(\bA)$  contains equal-height, $2s$-sparse vectors, for example, $\sum_{j=1}^{2s} \be_j$, in which case the recovery problem of an equal-height, $s$-sparse vector, say $\bz = \sum_{j=1}^{s} \overline{z} \be_j$, would essentially have at least another solution, namely $\bz' = - \sum_{j=s+ 1}^{2s}  \overline{z} \be_j$, as $R(\bz) = R(\bz')$. We therefore exclude the reconstruction of equal-height vectors under \eqref{iNSP:nonsep}.   

\begin{proposition}
\label{prop:non-sep}
\begin{enumerate}[\ \ \ i)]
\item The following methods:
\begin{center}
two-level $\ell_1$, sorted $\ell_1$ with $\beta_{s+1} > 0$
\end{center}
 are at least as good as $\ell_1$-minimization in recovering sparse vectors in the sense that these methods exactly reconstruct all $s$-sparse vectors under the null space property \eqref{NSP:l1}. 
 \item The following methods:
\begin{center}
$\ell_1 - \ell_2$, sorted $\ell_1$ with $\beta_{s+1} > \beta_{s}$
\end{center}
 are provably superior to $\ell_1$-minimization in recovering sparse vectors in the sense that these methods exactly reconstruct all $s$-sparse (except equal-height) vectors under the improved null space property \eqref{iNSP:nonsep}. 
 \end{enumerate}
\end{proposition}
\begin{proof} In this proof, for convenience, we often drop the zero components when denoting vectors in $\mathbb{R}^N$, for instance, $(z_1,\ldots,z_s,0,\ldots,0)$ with $z_i \ne 0, \forall\, 1\le i\le s$ and $s<N$ is simply represented as $(z_1,\ldots,z_s)$. Applying Theorem \ref{main_theorem}, we only need to show that: i) $R_{2\ell_1}$ and $R_{s\ell_1}$ with $\beta_{s+1} > 0$ satisfy \eqref{assump1} and ii) $R_{\ell_1 - \ell_2}$ and $R_{s\ell_1}$ with $\beta_{s+1} > \beta_{s}$ satisfy \eqref{assump2}. 
\begin{enumerate}
\item $R_{\ell_1 - \ell_2}$: For $z_1,\ldots, z_{s+ 1} > 0$, 
\begin{align*}
& \left(\sqrt{ z_1^2 + \ldots + z_{s+1}^2} + z_{s+1}\right)\sqrt{ z_1^2 + \ldots + z_{s}^2} > z_1^2 + \ldots + z_{s}^2
\\
& \qquad =  \left(\sqrt{ z_1^2 + \ldots + z_{s+1}^2} + z_{s+1}\right) \left(\sqrt{ z_1^2 + \ldots + z_{s+1}^2} - z_{s+1}\right), 
\end{align*}
thus, $z_{s+1} - \sqrt{ z_1^2 + \ldots + z_{s+1}^2}   > - \sqrt{ z_1^2 + \ldots + z_{s}^2} $. Adding $ z_1 + \ldots + z_{s}$ in both sides yields \eqref{assump2}. 

\item $R_{2\ell_1}$: We have
$$
R_{2\ell_1}(z_1,\ldots, z_{s+ 1}) \ge R_{2\ell_1}(z_1,\ldots, z_{s}) + \rho z_{s+1} > R_{2\ell_1}(z_1,\ldots, z_{s})
$$ 
for all $z_1,\ldots, z_{s+ 1} > 0$ and \eqref{assump1} is deduced. 

\item $R_{s\ell_1}$: Let us define $\bz = (z_1,\ldots, z_{s+1},0,\ldots,0)\in \mathbb{R}^N$ and assume $z_{s}$ and $z_{s+1}$ are the $T$-th and $t$-th largest components of $\bz$, i.e., $z_{s} = z_{[T]}$, $z_{s+1} = z_{[t]}$.   

\hspace{.2in}
Consider $\beta_{s+1} > 0$, we assume $t<s+1$ (the other case $t = s+1$ is trivial). For any $j$ with $t \le j \le s$, $  \beta_j z_{[j]} \ge \beta_{j} z_{[j+1]}$. At $j = s+1$, we estimate $\beta_{s+1} z_{[s+1]} > 0$. There follows 
\begin{align*}
& R_{s\ell_1}(z_1,\ldots, z_{s+1})  = \sum_{j=1}^{s+1} \beta_j z_{[j]} > \sum_{j=1}^{t-1} \beta_j z_{[j]} + \sum_{j=t}^s \beta_j  z_{[j+1]}  = R_{s\ell_1}(z_1,\ldots, z_{s}), 
\end{align*}
giving \eqref{assump1}.
 
\hspace{.2in}
Next, consider $\beta_{s+1} > \beta_{s}$. Without loss of generality, assume $t > T$ and also $t<s+1$ (the below argument also applies to $t = s+1$ with minor changes). At $j= t$, we estimate $\beta_t z_{[t]} \ge \beta_{T} z_{[t]}$. For all $j$ with $t+1 \le j\le s +1 $, $  \beta_j z_{[j]} \ge \beta_{j-1} z_{[j]}$. In particular, at $j = s+1$, the strict inequality holds, i.e., $  \beta_{s+1} z_{[s+1]} > \beta_{s} z_{[s+1]}$. Combining these facts and applying rearrangement inequality yield 
\begin{align*}
&  R_{s\ell_1}(z_1,\ldots, z_{s+1})   = \sum_{j=1}^{s+1} \beta_j z_{[j]}  > \sum_{\substack{ j=1 } }^{t-1} \beta_j z_{[j]} + \beta_{T}  z_{[t]} + \sum_{j = t+1}^{s+1}  \beta_{j-1} z_{[j]} 
  \\
&  = \sum_{\substack{ j=1 \\ j\ne T} }^{t-1} \beta_j z_{[j]} + \sum_{j = t+1}^{s+1}  \beta_{j-1} z_{[j]} + \beta_{T} ( z_{[T]} + z_{[t]})
\ge  R_{s\ell_1}(z_1,\ldots, z_{s-1}, z_s + z_{s+1})  . 
 \end{align*}
 We obtain \eqref{assump2} and complete the proof. 
\end{enumerate}
\end{proof}

\subsection{Proof of Theorem \ref{main_theorem}}
\label{sec:proof}
First, since $R$ is concave on $\mathcal{U}$, $R$ is also increasing on $\mathbb{R}^N$; see Lemma \ref{lem:incr_prop} and the discussion after. In what follows we denote 
\begin{align*}
\mcK_1 &:= \{\bv\in \mathbb{R}^N: &\ \|\bv_S\|_1 < \|\bv_{\oS}\|_1,\ \forall S\subset \{1,\ldots,N\}\, \text{ with }\, \#(S)\le s\}, 
\\
\mcK_2 &:= \{\bv\in \mathbb{R}^N \setminus \{\mathbf{0}\}: \!\!\!\!\!\!\!\!\!\!\!\!\!\!\!\!\!\! &\ \|\bv_S\|_1 \le \|\bv_{\oS}\|_1,\ \forall S\subset \{1,\ldots,N\} \, \text{ with }\, \#(S)\le s\}. 
\end{align*}
Recall that if $\bA $ satisfies \eqref{NSP:l1} or correspondingly \eqref{iNSP:nonsep}, then $\ker(\bA) \setminus \{\textbf{0}\}\subset \mcK_1$ or $\ker(\bA)\setminus\{\textbf{0}\} \subset \mcK_2$ respectively. Let $\bx$ be a fixed $s$-sparse vector in $\mathbb{R}^N$, then $\bz = \bx$ solves $\bA\bz = \bA \bx$ and any other solution of this system can be written as $\bz = \bx + \bv$ with some $\bv \in \ker(\bA)\setminus\{\textbf{0}\} \subset \mathcal{K}_{1|2}$. We will show in the next part that:
\begin{itemize}
\vspace{.03in}
\item if \eqref{assump1} holds then  
$$
R(\bx + {\bv}) > R(\bx),\ \forall \bv\in \mathcal{K}_1; \ \text{ and } 
$$
\item if \eqref{assump2} holds and $\bx$ is not an equal-height vector then  
$$
R(\bx + {\bv}) > R(\bx),\ \forall \bv\in \mathcal{K}_2,
$$
\vspace{.03in}
\end{itemize}
 thus $\bx$ is the unique solution to \eqref{problem:PR} assuming either: (i) \eqref{assump1} and \eqref{NSP:l1}; or (ii) \eqref{assump2} and \eqref{iNSP:nonsep} and $\bx$ is not an equal-height vector.  
   
Since $R$ is symmetric, without loss of generality, we assume 
\begin{align}
&\bx = (x_1,\ldots,x_s,0,\ldots,0) , \mbox{ where }x_i \ge 0,\, \forall 1\le i\le s.   \label{define_x}
\end{align}
For every $\bv \in \cK_{1|2}$, there exists $\tilde{\bv} \in \cK_{1|2}$  that 
$$
R(\bx + \tilde{\bv})  \le R(\bx + {\bv} ), 
$$
and the $s$ first components of $\tilde{\bv}$ are nonpositive, i.e., 
$
\tilde{v}_i \le 0,\, \forall 1\le i\le s. 
$
Indeed, let ${\bv} = (v_1,\ldots,v_N)$ be a vector in $\mathcal{K}_{1|2}$. For any $1\le i \le s$, if ${v}_i > 0$, we flip the sign of ${v}_i$, i.e., replacing ${v}_i $ by $-{v}_i $. Denoting the newly formed vector by $\tilde{\bv}$, then $\tilde{\bv}\in \mcK_{1|2}$ and 
$$
|(\bx + \tilde{\bv})_i | = |x_i -{v}_i | \le x_i + {v}_i = |(\bx + {\bv})_i|, 
$$
for every $ i $ where the sign is flipped, while $(\bx + \tilde{\bv})_i = (\bx + {\bv})_i$, for other $i$. By the increasing property of $R$, $\tilde{\bv}$ satisfies 
$
R(\bx + \tilde{\bv})  \le R(\bx + {\bv} ). 
$
The first $s$ components of $\tilde{\bv}$ are nonpositive by definition. 

It is therefore enough to consider $\bv\in \cK_{1|2}$ represented as
\begin{align}
\label{def:v}
{\bv} = (- {a}_1,\ldots,- {a}_s,{b}_1,\ldots,{b}_t,0,\ldots,0), 
\end{align}
where $\{{a}_i\},\, \{{b_j}\}$ are correspondingly nonnegative and positive sequences. We denote by $E({\bv})$ the multiset $ \{{a}_i:i\in \overline{1,s}\}\cup  \{{b}_j: j\in \overline{1,t}\}$, by $U({\bv})$ the multiset containing $s$ largest elements in $E(\bv)$ and $L(\bv) := E(\bv)\setminus U(\bv)$. Let 
\begin{align*}
&\overline{b}(\bv) := \min U(\bv), \mbox{ and}
\\
&\sigma(\bv)\mbox{ and }\lambda(\bv)\mbox{ be the sum of all elements in }U(\bv)\mbox{ and }L(\bv) \mbox{ respectively}.
\end{align*}
Then ${\bv} \in \mathcal{K}_1$ (or ${\bv} \in \mathcal{K}_2$ correspondingly) if and only if $\sigma(\bv)\! < \lambda(\bv)$ ($\sigma(\bv)\! \le \lambda(\bv)$ and $\bv \ne \mathbf{0}$ respectively). Also, note that $t\ge s+1 $ for all ${\bv} \in \mathcal{K}_1$, and $t \ge s $ for all ${\bv} \in \mathcal{K}_2$ with $t=s$ occurring only if $\bv$ is an equal-height vector.  Below we consider two cases.
\vspace{.03in}

\textbf{Case 1:} Assume \eqref{assump1} and $\bv\in \cK_1$ as in \eqref{def:v}. 

We will show there exists $\widehat{\bv}\in \mathcal{K}_1 $  that  
\begin{align*}
 \widehat{\bv} = & \, (- {a}_1,\ldots,- {a}_s,\underbrace{\overline{b},\ldots,\overline{b}}_{T-1},b_{T},0,\ldots,0),\ \text{for some }T\ge s+1, \ 0<  b_{T} \le \overline{b} \equiv \overline{b}(\bv), 
\\
 & \text{and }\ \   \qquad \quad R(\bx + \widehat{\bv})  \le R(\bx + {\bv} ). 
\end{align*}

First, we replace all components $b_j$ in $U(\bv)$ (which satisfy $b_j\ge \ob$) by $\ob$ and subtract a total amount $\sum_{b_j \in U(\bv)} ( b_j - \overline{b})$ from $b_{j}$'s in $L(\bv)$. It is possible to form an elementwise nonnegative vector (referred to as $\bv'$) with this step, since
\begin{align}
\label{mp:est1}
\sum_{b_j \in U{(\bv)}} ( b_j - \overline{b}) < \sum_{b_j \in L(\bv)}  b_j .  
\end{align}
Indeed, one has $\sum\limits_{a_i \in L(\bv)}  a_i \le s\ob \le  \sum\limits_{a_i \in U(\bv)}  a_i + \sum\limits_{b_j \in U(\bv)}  \ob$. Combining with $\sigma(\bv)\! < \lambda(\bv)$, it gives  
$$
 \sum\limits_{a_i \in U(\bv)}  a_i +  \sum\limits_{b_j \in U(\bv)}  b_j < \sum\limits_{a_i \in L(\bv)}  a_i +  \sum\limits_{b_j \in L(\bv)}  b_j \le  \sum\limits_{a_i \in U(\bv)}  a_i + \sum\limits_{b_j \in U(\bv)}  \ob +  \sum\limits_{b_j \in L(\bv)}  b_j, 
$$
yielding \eqref{mp:est1}. We remark that $\bv' \in \cK_1$, since  
$$
\sigma(\bv')\! = \sigma(\bv)  - \sum_{b_j \in U(\bv)} ( b_j - \overline{b}) < \lambda(\bv) - \sum_{b_j \in U(\bv)} ( b_j - \overline{b}) = \lambda(\bv'). 
$$
On the other hand, by the construction, the magnitudes of coordinates of $\bx + \bv'$ are less than or equal to that of $\bx + \bv$. The increasing property of $R$ gives $R(\bx + {\bv}')  \le R(\bx + {\bv} )$. 

Now, representing $\bv'$ as 
$$
\bv' =  (- {a}_1,\ldots,- {a}_s,{b}'_1,\ldots,{b}'_{t'},0,\ldots,0), 
$$
we observe $\sigma(\bv') \ge s\ob$, as $\bv'$ has at least $s$ components whose magnitudes are not less than $\ob$. Thus, $\sum_{j =1}^{t'} b'_{j} \ge \lambda(\bv') > \sigma(\bv') \ge s\ob$ and there exist $T \ge s+1 $ and $b_T \in (0,\ob]$ such that  
$(T-1)\ob + b_T= \sum_{j =1}^{t'} b'_{j}$. We define 
$$
 \widehat{\bv} =  \, (- {a}_1,\ldots,- {a}_s,\underbrace{\overline{b},\ldots,\overline{b}}_{T-1},b_{T},0,\ldots,0). 
$$
One has $U(\widehat{\bv}) = U(\bv')$, which implies $\sigma(\widehat{\bv}) = \sigma(\bv')$ and $\lambda(\widehat{\bv}) = \lambda(\bv')$. Then, $\widehat{\bv} \in \cK_1$ can be deduced from the fact that $\bv' \in \mathcal{K}_1$. As $b'_{j} \le \ob\ \, \forall 1\le j\le t' $, it is easy to see {$|\bx + \bv'| \prec  |\bx + \widehat{\bv} |$}. From Lemma \ref{lemma:majorize}, there follows $ R( \bx + \widehat{\bv}) \le R(\bx + \bv')$.   

We proceed to prove $R(\bx) < R(\bx + \widehat{\bv})$. 
Let $\oa = \sum_{i=1}^s a_i / s$. If $\oa = 0$, the assertion can be deduced easily from the increasing property of $R$ and \eqref{assump1}. Let us consider $\oa >0$. Since $\widehat{\bv} \in \mathcal{K}_1$, 
$
(T-1)\ob + b_T > s\oa. 
$
There exists $0< \kappa < 1 $ such that 
$$
(T-1)\kappa \ob + \kappa b_T > s\oa > (T-1)\kappa \ob. 
$$
We write 
$
(T-1)\kappa \ob + \kappa b_T = s\oa + a' = a_1 + \ldots + a_s + a' , 
$
then $ 0 < a' < \kappa b_T$. Also, note that $\oa > \kappa \ob$, as $T - 1 \ge s$. Denoting
\begin{gather}
\label{maj_series}
\begin{aligned}
\bz_1 &=  (|x_1- {a}_1|,\ldots,| x_s - {a}_s|,\underbrace{\kappa\overline{b},\ldots,\kappa\overline{b}}_{T-1},\kappa b_{T},0,\ldots,0),
\\
\bz_2 &=  (|x_1- {a}_1|,\ldots,| x_s - {a}_s|,\underbrace{\oa,\ldots,\oa}_{s},a',0,\ldots,0),
\\
\bz_3 & =  (|x_1- {a}_1|,\ldots,| x_s - {a}_s|,{a_1,\ldots,a_s},a',0,\ldots,0),
\\
\bz_4 & =  (|x_1- {a}_1| + a_1,\ldots,| x_s - {a}_s|+a_s,a',0,\ldots,0),
\\
\bx' & = ({x_1,\ldots,x_s},a',0,\ldots,0), 
\end{aligned}
\end{gather}
there holds 
$
\bz_1 \prec \bz_2 \prec \bz_3 \prec \bz_4. 
$
Applying Lemma \ref{lemma:majorize} yields
\begin{align}
\label{mp:est2}
R(\bz_1) \ge R(\bz_2) \ge R(\bz_3) \ge R(\bz_4). 
\end{align}
On the other hand, $\bx + \widehat{\bv} \ge \bz_1$ and $\bz_4 \ge \bx'$, thus 
\begin{align}
\label{mp:est3}
R(\bx + \widehat{\bv}) \ge R(\bz_1)\ \text{ and }\ R(\bz_4) \ge R(\bx'). 
\end{align}
We have from \eqref{assump1} that
\begin{align}
\label{mp:est4}
R(\bx') > R(\bx).
\end{align}
Combining \eqref{mp:est2}--\eqref{mp:est4} gives $R(\bx + \widehat{\bv}) > R(\bx)$, as desired. 

\textbf{Case 2:} Assume \eqref{assump2} and $\bv\in \cK_2$ as in \eqref{def:v} and $\bx$ is not an equal-height vector. 
Following the arguments in Case 1, there exists $\widehat{\bv}\in \mathcal{K}_2 $  that  
\begin{align*}
 \widehat{\bv} = & \, (- {a}_1,\ldots,- {a}_s,\underbrace{\overline{b},\ldots,\overline{b}}_{T-1},b_{T},0,\ldots,0),\ \text{for some }T\ge s+1, \ 0 \le  b_{T} < \overline{b} \equiv \overline{b}(\bv), 
\\
 & \text{and }\ \   \qquad \quad R(\bx + \widehat{\bv})  \le R(\bx + {\bv} ).
\end{align*}
Note that here $b_T \in [0,\ob)$, implying $\#(\supp(\widehat{\bv})) \ge 2s$ (rather than $b_T \in (0,\ob]$ and $\#(\supp(\widehat{\bv})) \ge  2s+1$ as in previous case). 

First, if $\#(\supp(\widehat{\bv})) = 2s$, then $\widehat{\bv}$ must be an equal-height vector: 
\begin{gather*}
\begin{aligned}
\widehat{\bv} &= \, (\underbrace{- \ob,\ldots,- \ob}_s,\underbrace{\overline{b},\ldots,\overline{b}}_{s},0,\ldots,0),
\\
\bx + \widehat{\bv} &= \, ({x_1- \ob,\ldots,x_s- \ob},\underbrace{\overline{b},\ldots,\overline{b}}_{s},0,\ldots,0).
\end{aligned}
\end{gather*}
We denote 
\begin{gather*}
\begin{aligned}
\bz_5 &=  (|x_1- {\ob}| + \ob,\ldots,| x_s - {\ob}| + \ob ,0,\ldots,0).
\end{aligned}
\end{gather*}
Since $\bx$ is not an equal-height vector, $|x_{i} - \ob| \ne 0$ for some $1\le i \le s$. Lemma \ref{lemma:majorize} and assumption \eqref{assump2} give 
$
R(\bz_5) < R(\bx + \widehat{\bv}). 
$
It is easy to see $\bx \le \bz_5$, therefore, 
$
R(\bx) \le R(\bz_5), 
$
and we arrive at $R(\bx) < R(\bx + \widehat{\bv})$.

Otherwise, if $\#(\supp(\widehat{\bv})) \ge 2s +1$, then 
$
(T-1) \ob +  b_T > s \ob. 
$
Let us again denote $\oa = \sum_{i=1}^s a_i / s$ and consider $\oa >0$.  
Since $\widehat{\bv}\in \mathcal{K}_2 $,   
$
(T-1)\ob + b_T \ge s\oa, 
$
and we can find $0< \kappa \le 1 $  that 
$$
(T-1)\kappa \ob + \kappa b_T \ge s\oa > s \kappa \ob. 
$$
We write 
$
(T-1)\kappa \ob + \kappa b_T = s\oa + a' , 
$
then $ a' \ge 0$ and $\oa > \kappa \ob$. 
Denoting $\bz_1, \bz_2, \bz_3, \bz_4$ and $\bx'$ as in \eqref{maj_series}, similarly to Case 1, there holds 
$$
R(\bx) \le R(\bx') \le R(\bz_4) \le R(\bz_3) \le R(\bz_2) \le R(\bz_1) \le R(\bx + \widehat{\bv})
$$
due to 
$
\bz_1 \prec \bz_2 \prec \bz_3 \prec \bz_4, \, 
$
$
\bz_4 \ge \bx' \ge \bx
$
 and
 $
 \bz_1 \le \bx + \widehat{\bv}. 
 $
We show that the strict inequality must occur somewhere in the chain. If $a' > 0$, we have from \eqref{assump2} that 
 $
 R(\bx) < R(\bx'). 
 $
 Otherwise, if $a'=0$, then $\#(\supp(\bz_1)) \ge s + 1 $ and $\#(\supp(\bz_4)) \le s $. Applying Lemma \ref{lemma:majorize} and assumption \eqref{assump2} gives 
$
 R(\bz_4) < R(\bz_1).
$ 
This concludes the proof. $\square$

\section{Exact recovery of sparse signals via separable penalties}
\label{sec:separable}
With the addition of separable property, the path to establish the null space condition for nonconvex minimizations is much simpler. To highlight the difference between the separable and non-separable cases, in this section, we revisit the exact recovery of sparse signals assuming the penalty is concave and separable. The discussion herein is applicable for the following well-known regularizations. 
\begin{example}[Concave, separable and symmetric penalties]
\label{example:sep_func}
\begin{enumerate}
\itemsep10pt
\item $\ell_p$ norm with $0< p<1$: \ \  \ $R_{\ell_p}(\bm{z}) = \|\bm{z}\|_p^p$, \cite{Chartrand07,FoucartLai09}.
\item SCAD:\qquad \qquad \qquad \qquad \,\ \  	$R_{\text{\fontfamily{cmss}\selectfont SCAD}}(\bz) = \sum_{j=1}^N r_{\text{\fontfamily{cmss}\selectfont SCAD}} (z_j)$, \cite{FanLi01}.

 \vspace{.1cm}
Here, $r_{\text{\fontfamily{cmss}\selectfont SCAD}}(z_j) = 
\begin{cases}
a_1 |z_j|,\, &\text{ if } |z_j|< a_1,
\\
- \frac{a_1|z_j|^2 - 2 a_1 a_2 |z_j| + a_1^3}{2(a_2 - a_1)} 	 &\text{ if } a_1 \le |z_j| \le a_2, 
\\
\frac{a_1 a_2 + a_1^2}{2} &\text{ if } |z_j|> a_2. 
\end{cases}
$
\item Transformed $\ell_1$: \qquad \qquad \,  $R_{t\ell_1}(\bz) = \sum_{j=1}^N \rho_a(z_j)$, \cite{LvFan09,ZhangXin16}. 
 
  \vspace{.1cm}
Here, $\rho_a(z_j) = \frac{(a+1)|z_j|}{a + |z_j|},\, \forall z_j\in \mathbb{R}$ with $a\in (0,\infty)$. 
\item Capped $\ell_1$: \qquad \qquad \quad \quad \, \ $R_{c\ell_1}(\bm{z}) = \sum\limits_{j=1}^N \min\{|z_j|,\alpha\}$, {\cite{NIPS2008_3526}}.
\end{enumerate}
\end{example}

The first key difference is that with the separable property, one can obtain the subadditivity of penalty functions.
\begin{lemma}
\label{prop:triangle_ineq}
Let $R$ be a map from $\mathbb{R}^N$ to $[0,\infty)$ satisfying $R(z_1,\ldots,z_N) = R(|z_1|,\ldots,|z_N|),\ \forall \bm{z} = (z_1,\ldots,z_N)\in\mathbb{R}^N$. If $R$ is separable on $\mathbb{R}^N$ and concave on $\cU$, then $R$ is subadditive on $\mathbb{R}^N$: 
$$
R(\bz + \bz') \le R(\bz) + R(\bz'),\ \forall \bz,\bz' \in \mathbb{R}^N. 
$$
\end{lemma}
\begin{proof}
The assertion can be implied from its univariate version, see, e.g., \cite[Problem II.5.12]{Bhatia97} and the separable property of $R$.
\end{proof}

In this case, it is possible and also natural to show \eqref{iNSP:sep} to be the \textit{necessary and sufficient} condition for the uniform, exact recovery via fixed support setting. The following theorem extends several results for specific penalties, e.g., $\ell_p$ \cite{FoucartLai09}, weighted $\ell_1$ \cite{RW15}. Similar result was proved in \cite{GribonvalNielsen07}. 

\begin{theorem}
\label{theorem_separable}
Let $\bA$ be an $m\times N$ real matrix. Consider the problem \eqref{problem:PR}, where $R$ is a function from $\mathbb{R}^N$ to $[0,\infty)$ satisfying $R(z_1,\ldots,z_N) = R(|z_1|,\ldots,|z_N|),\ \forall \bm{z} = (z_1,\ldots,z_N)\in\mathbb{R}^N$, separable on $\mathbb{R}^N$, concave on $\cU$ and $R( \mathbf{0}) = 0$. Then every $s$-sparse vector $\bm{x}\in \mathbb{R}^N$ is the unique solution to \eqref{problem:PR} if and only if 
the generalized null space property \eqref{iNSP:sep} is satisfied.
\end{theorem}

\begin{proof}
It is enough to show that for index set $S\subset \{1,\ldots,N\}$, every vector $\bm{x}\in \mathbb{R}^N$ supported in $S$ is the unique solution to \eqref{problem:PR} if and only if 
$$
\ker(\bA)\setminus\{\mathbf{0}\}\subset \{ \bv \in \mathbb{R}^N: 
R(\bv_S)  < R(\bv_{\overline{S}})\}.
$$

First, assume that every vector $\bm{x}\in \mathbb{R}^N$ supported in $S$ is the unique solution to \eqref{problem:PR}. For any $\bv\in \ker(\bA)\setminus \{\mathbf{0}\}$, $\bm{v}_S$ is the unique minimizer of $R(\bz)$ subject to $\bA \bz = \bA \bv_S$. Observe that $\bA(-\bv_{\oS}) = \bA\bv_S$ and $-\bv_{\oS} \ne \bv_S$ (since $\bv\ne \mathbf{0}$), we have $R(\bv_S) < R(\bv_{\oS})$. 

Conversely, assume $\ker(\bA)\setminus\{\mathbf{0}\}\subset \{ \bv \in \mathbb{R}^N: 
R(\bv_S)  < R(\bv_{\overline{S}})\}$. Let $\bx\in \mathbb{R}^N$ be a vector supported in $S$. Any other solution to $\bA \bz = \bA \bx$ can be represented as $\bz = \bx + \bv$ with $\bv \in \ker(\bA)\setminus\{\textbf{0}\}$. We have by the separable property of $R$ and Lemma \ref{prop:triangle_ineq}
\begin{gather}
\label{sec:sepa:est1}
\begin{aligned}
R(\bx + \bv)   &= R(\bx + \bv_S) + R(\bv_{\oS}) \ge {R(\bx) - R(\bv_S)} + R(\bv_{\oS}) > R(\bx),
\end{aligned}
\end{gather}
thus $\bx$ is the unique solution to \eqref{problem:PR}.
\end{proof}

{
\begin{remark}
\eqref{iNSP:sep} is actually a necessary condition for the exact recovery of every $s$-sparse vector even when the separable and concave property on $R$ is removed, thus applicable to a very general class of penalty functions. For this condition to become sufficient, the separable assumption on $R$ is critical. In estimate \eqref{sec:sepa:est1}, we utilize this assumption in two ways: splitting $R(\bx+ \bv)$ into $R(\bx + \bv_S) + R(\bv_{\oS})$, and bounding $R(\bx + \bv_S) \ge {R(\bx) - R(\bv_S)} $ via the subadditivity of $R$. Without the separable property, a concave penalty may not be subadditive. For example, consider $R_{\ell_1 - \ell_2}:\mathbb{R}^2 \to [0,\infty)$, corresponding to $\ell_1-\ell_2$ regularization. $R_{\ell_1 -\ell_2}$ is not separable and also not subadditive, as one has  
$$
2 - \sqrt{2} = R_{\ell_1 -\ell_2}(1,1) > R_{\ell_1 -\ell_2}(1,0) + R_{\ell_1 -\ell_2}(0,1) = 0. 
$$
In this case, the analysis for the sufficient condition follows a significantly different and more complicated path (see Section \ref{sec:nonseparable}).
\end{remark}

It is worth emphasizing that the necessary and sufficient condition \eqref{iNSP:sep} for concave and separable penalties is not necessarily less demanding than \eqref{NSP:l1}, as shown in the below example. 

\begin{example}
\label{example:nonsymmetric}
Consider the underdetermined system 
$
\bA\bz = \bA\bx,
$
where the matrix $\bA\in \mathbb{R}^{4\times 5}$ is defined as 
\[ \bA =  \left( \begin{array}{ccccc}
 1 & \frac12 & \frac94  & 0 & 0 \\
 1 & -\frac12 & 0 & \frac34 & 0 \\
 0 & 1 & 0 & 0 & \frac43 \\
 1 & -1 & 0 & 0 & 0 
 \end{array} \right).  
 \]
 As $\ker(\bA) = (-t,-t,2t/3,2t/3,3t/4)^{\top}$ satisfies \eqref{NSP:l1} with $N =5$ and $s = 2$, one can successfully reconstruct all $2$-sparse vectors using $\ell_1$ minimization. Consider the weighted $\ell_1$ regularization 
 $$
 R_{w\ell_1} (\bz) =  4 |z_1| + 3 |z_2| + |z_3| + |z_4| + |z_5|, 
 $$
 which is normally perceived as a convex penalty but also concave according to Definition \ref{def:penalty_prop}. $R_{w\ell_1}$ is not symmetric, and we can see that not every $2$-sparse vector can be recovered using this penalty, especially those {whose first two entries are nonzero}. For instance, if $\bx = (1,1,0,0,0)^{\top}$, all solutions to $
\bA\bz = \bA\bx
$ can be represented as $\bz = (1-t,1-t,2t/3,2t/3,3t/4)$. Among these, the solution that minimizes $R_{w\ell_1}(\bz)$ is $(0,0,2/3,2/3,3/4)$, different from $\bx$ and not the sparsest one. Nevertheless, it should be noted that weighted $\ell_1$ minimization, with appropriate choices of weights can be a very efficient approach that outperforms $\ell_1$,  
in the case when a priori knowledge of (the structure of) the support set is available, see, e.g., \cite{FMSY12,YB13,ChkifaDexterTranWebster15,Adcock15b}. 
\end{example}

Finally, \eqref{iNSP:sep} is truly less demanding than \eqref{NSP:l1} if, in addition, the penalty function is symmetric. The following result verifies that separable, concave, and symmetric regularizations are superior to $\ell_1$ minimization in sparse, uniform recovery. 
}

\begin{proposition}
Let $N>1,\, s\in \mathbb{N}$ with $1 \le  s <  N/2 $. Assume $R$ is a function from $\mathbb{R}^N$ to $[0,\infty)$ satisfying $R(z_1,\ldots,z_N)$ $= R(|z_1|,\ldots,|z_N|),\ \forall \bm{z} = (z_1,\ldots,z_N)\in\mathbb{R}^N$, separable and symmetric on $\mathbb{R}^N$, and concave on $\cU$. Also, $R( \mathbf{0}) = 0$ and $R(\bz) > 0,\, \forall \bz \ne \mathbf{0}$. Then 
\begin{align*}
& \bigg\{ \bv \in \mathbb{R}^N: 
\|\bv_S\|_1  < \|\bv_{\overline{S}}\|_1,\, \forall S\subset \!\{1,\ldots,N\}\mbox{ with }\#(S)\le s\bigg\}
 \subset 
\\
& \qquad \qquad \bigg\{ \bv \in \mathbb{R}^N: \,
R(\bv_S)  < R(\bv_{\overline{S}}),\, \forall S\subset \!\{1,\ldots,N\}\mbox{ with }\#(S)\le s\bigg\}. 
\end{align*}
Consequently, \eqref{iNSP:sep} is less demanding than \eqref{NSP:l1}. 
\end{proposition}
\begin{proof}
Let $\bv \in \mathbb{R}^N$ satisfy $\|\bv_S\|_1  < \|\bv_{\overline{S}}\|_1,\, \forall S\subset \!\{1,\ldots,N\}${ with }$\#(S)\le s$. We denote by $S^\star$ the set of indices of $s$ largest components of $\bv$ (in magnitude), then $\|\bv_{S^{\star}}\|_1  < \|\bv_{\overline{S^{\star}}}\|_1$. Since $R$ is concave on $\cU$, by Lemma \ref{lem:incr_prop}, $R$ is increasing on $\cU$. It is enough to prove that $R(\bv_{S^{\star}})  < R(\bv_{\overline{S^{\star}}})$. Let $\alpha = \|\bv_{\overline{S^{\star}}}\|_1 - \|\bv_{S^{\star}}\|_1 > 0$ and $j$ be an index in $\overline{S^{\star}}$, we define 
$
\tilde{\bv} = \bv_{{S^{\star}}} + \alpha \be_j. 
$
Since $R$ is separable and $R(\bz) > 0,\, \forall \bz \ne \mathbf{0}$, one has 
\begin{align}
R(\tilde{\bv}) = R( \bv_{{S^{\star}}} ) +  R( \alpha \be_j ) > R( \bv_{{S^{\star}}} ) . \label{sec:sepa:est2}
\end{align} 
On the other hand, $\|\tilde{\bv}\|_1 =  \|{\bv_{\overline{S^{\star}}}}\|_1$ and, by the definition of $S^{\star}$, any nonzero component of $\tilde{\bv}$ (with the possible exception of the $j$-th one) is larger than any component of $\bv_{\overline{S^{\star}}}$. This yields $\tilde{\bv} \succ  {\bv_{\overline{S^{\star}}}}$. Applying Lemma \ref{lemma:majorize}, there follows $R(\tilde{\bv}) \le R({\bv_{\overline{S^{\star}}}})$. Combining with \eqref{sec:sepa:est2}, the proposition is concluded. 
\end{proof}

\section{Concluding remarks}
\label{sec:conclusion}
In this effort, we establish theoretical, generalized sufficient conditions for the uniform recovery of sparse signals via concave, non-separable and symmetric regularizations. These conditions are proved less restrictive than the standard null space property for $\ell_1$ minimization, thus verifying that concave, non-separable and symmetric penalties are better than or at least as good as $\ell_1$ in enhancing the sparsity of the solutions. Our work unifies and improves existing NSP-based conditions developed for several specific penalties, and also provides first theoretical recovery guarantees in some cases. 

Extending the present results to the more practical scenarios, which allows measurement errors and compressible (i.e., only close to sparse) signals  is the next logical step. In particular, an important open question is: are concave and symmetric regularizations still provably better than $\ell_1$ in uniform recovery, when taking noise and sparsity defect into account? Also, the general sufficient conditions for non-separable penalties, established herein from that of $\ell_1$, may be suboptimal for specific penalties. It will be interesting to investigate the specialized and optimized conditions in such cases. Finally, while the advantage of nonconvex minimizations over $\ell_1$ in terms of null space property is obvious, how this advantage reflects itself in sample complexity is unclear to us and a topic for future work. 

\bibliographystyle{spmpsci}      
\bibliography{nonconvex_R2}   

\begin{thebibliography}{10}
\providecommand{\url}[1]{{#1}}
\providecommand{\urlprefix}{URL }
\expandafter\ifx\csname urlstyle\endcsname\relax
  \providecommand{\doi}[1]{DOI~\discretionary{}{}{}#1}\else
  \providecommand{\doi}{DOI~\discretionary{}{}{}\begingroup
  \urlstyle{rm}\Url}\fi

\bibitem{Adcock15b}
Adcock, B.: Infinite-dimensional compressed sensing and function interpolation.
\newblock Found Comput Math  (2017)

\bibitem{Bhatia97}
Bhatia, R.: Matrix Analysis, \emph{Graduate Texts in Mathematics}, vol. 169.
\newblock Springer (1997)

\bibitem{CRT06}
Cand\`{e}s, E., Romberg, J., Tao, T.: Robust uncertainty principles: exact
  signal reconstruction from highly incomplete frequency information.
\newblock IEEE Trans. Inform. Theory \textbf{52}(1), 489--509 (2006)

\bibitem{Chartrand07}
Chartrand, R.: Exact reconstruction of sparse signals via nonconvex
  minimization.
\newblock IEEE Signal Processing Letters \textbf{14}(10), 707--710 (2007)

\bibitem{ChkifaDexterTranWebster15}
Chkifa, A., Dexter, N., Tran, H., Webster, C.: Polynomial approximation via
  compressed sensing of high-dimensional functions on lower sets.
\newblock Math. Comp.  (2017)

\bibitem{CohenDahmenDeVore09}
Cohen, A., Dahmen, W., DeVore, R.: Compressed sensing and best $k$-term
  approximation.
\newblock J. Amer. Math. Soc. \textbf{22}, 211--231 (2009)

\bibitem{DexterTranWebster17}
Dexter, N., Tran, H., Webster, C.: {A mixed $\ell_1$ regularization approach
  for sparse simultaneous approximation of parameterized PDEs}.
\newblock submitted  (2018)

\bibitem{Donoho06}
Donoho, D.L.: Compressed sensing.
\newblock IEEE Trans. Inform. Theory \textbf{52}(4), 1289--1306 (2006)

\bibitem{DO11}
Doostan, A., Owhadi, H.: A non-adapted sparse approximation of pdes with
  stochastic inputs.
\newblock Journal of Computational Physics \textbf{230}, 3015--3034 (2011)

\bibitem{EsserLouXin13}
Esser, E., Lou, Y., Xin, J.: A method for finding structured sparse solutions
  to nonnegative least squares problems with applications.
\newblock SIAM J. Imaging Sci. \textbf{6}(4), 2010--2046 (2013)

\bibitem{FanLi01}
Fan, J., Li, R.: Variable selection via nonconcave penalized likelihood and its
  oracle properties.
\newblock Journal of the American Statistical Association \textbf{96}(456),
  1348--1360 (2001)

\bibitem{FoucartLai09}
Foucart, S., Lai, M.J.: Sparsest solutions of underdetermined linear systems
  via $\ell_q$-minimization for $0<q\le 1$.
\newblock Appl. Comput. Harmon. Anal. pp. 395--407 (2009)

\bibitem{FouRau13}
Foucart, S., Rauhut, H.: A Mathematical Introduction to Compressive Sensing.
\newblock Applied and Numerical Harmonic Analysis. Birkh{\"a}user (2013)

\bibitem{FMSY12}
Friedlander, M., Mansour, H., Saab, R., Yilmaz, O.: Recovering compressively
  sampled signals using partial support information.
\newblock IEEE Transactions on Information Theory \textbf{58}(2), 1122--1134
  (2012)

\bibitem{GribonvalNielsen07}
Gribonval, R., Nielsen, M.: {Highly sparse representations from dictionaries
  are unique and independent of the sparseness measure}.
\newblock Appl. Comput. Harmon. Anal. \textbf{22}, 335--355 (2007)

\bibitem{HastieTibshiraniWainwright15}
Hastie, T., Tibshirani, R., Wainwright, M.: Statistical Learning with Sparsity,
  \emph{Monographs on Statistics and Applied Probability}, vol. 143.
\newblock Taylor \& Francis Group, LLC (2015)

\bibitem{HuangLiuShi-et-al15}
Huang, X., Liu, Y., Shi, L., Van~Huffel, S., Suykens, J.: Two-level $\ell_1$
  minimization for compressed sensing.
\newblock Signal Processing \textbf{108}(459-475) (2015)

\bibitem{HuangShiYan15}
Huang, X., Shi, L., Yan, M.: Nonconvex sorted $\ell_1$ minimization for sparse
  approximation.
\newblock Journal of Operations Research Society of China \textbf{3}(207-229)
  (2015)

\bibitem{LvFan09}
Lv, J., Fan, Y.: A unified approach to model selection and sparse recovery
  using regularized least squares.
\newblock The Annals of Statistic \textbf{37}(6A), 3498--3528 (2009)

\bibitem{MarshallOlkinArnold11}
Marshall, A., Olkin, I., Arnold, B.: Inequalities: Theory of Majorization and
  Its Applications, 2nd edn.
\newblock Springer Series in Statistics. Springer (2011)

\bibitem{RS14}
Rauhut, H., Schwab, C.: Compressive sensing {P}etrov-{G}alerkin approximation
  of high dimensional parametric operator equations.
\newblock Math. Comp. \textbf{86}, 661--700 (2017)

\bibitem{RW15}
Rauhut, H., Ward, R.: Interpolation via weighted $\ell_1$-minimization.
\newblock Applied and Computational Harmonic Analysis \textbf{40}(2), 321--351
  (2016)

\bibitem{ShenPanZhu12}
Shen, X., Pan, W., Zhu, Y.: Likelihood-based selection and sharp parameter
  estimation.
\newblock J Am Stat Assoc. \textbf{107}(497), 223--232 (2012)

\bibitem{WangYin10}
Wang, Y., Yin, W.: Sparse signal reconstruction via iterative support
  detection.
\newblock SIAM J. Imaging Sci. \textbf{3}(3), 462--491 (2010)

\bibitem{YanShinXiu16}
Yan, L., Shin, Y., Xiu, D.: Sparse approximation using $\ell_1-\ell_2$
  minimization and its applications to stochastic collocation.
\newblock SIAM J. Sci. Comput. \textbf{39}(1), A229--A254 (2017)

\bibitem{YinEsserXin14}
Yin, P., Esser, E., Xin, J.: Ratio and difference of $\ell_1$ and $\ell_2$
  norms and sparse representation with coherent dictionaries.
\newblock Communications in Information and Systems \textbf{14}(2), 87--109
  (2014)

\bibitem{YinLouHeXin15}
Yin, P., Lou, Y., He, Q., Xin, J.: Minimization of $\ell_{1-2}$ for compressed
  sensing.
\newblock SIAM J. Sci. Comput. \textbf{37}(1), A536--A563 (2015)

\bibitem{YinXin16}
Yin, P., Xin, J.: Iterative $\ell_1$ minimization for non-convex compressed
  sensing.
\newblock submitted  (2016)

\bibitem{YB13}
Yu, X., Baek, S.: Sufficient conditions on stable recovery of sparse signals
  with partial support information.
\newblock IEEE Signal Processing Letters \textbf{20}(5), 539--542 (2013)

\bibitem{ZhangXin16}
Zhang, S., Xin, J.: Minimization of transformed $l_1$ penalty: Theory,
  difference of convex function algorithm, and robust application in compressed
  sensing.
\newblock arxiv:1411.5735v3[cs.IT]  (2016)

\bibitem{NIPS2008_3526}
Zhang, T.: Multi-stage convex relaxation for learning with sparse
  regularization.
\newblock In: D.~Koller, D.~Schuurmans, Y.~Bengio, L.~Bottou (eds.) Advances in
  Neural Information Processing Systems 21, pp. 1929--1936. Curran Associates,
  Inc. (2009)

\bibitem{ZhangT10}
Zhang, T.: Analysis of multi-stage convex relaxation for sparse regularization.
\newblock Journal of Machine Learning Research \textbf{11}, 1081--1107 (2010)

\end{thebibliography}
\end{document}